\documentclass[a4paper,12pt]{amsart}
\usepackage[english]{babel}

\usepackage{amsmath}
\usepackage{amsthm}
\usepackage{amssymb}
\usepackage{amsfonts}
\usepackage{mathtools}
\usepackage{extarrows}
\usepackage{bbm}

\usepackage[
    a4paper,
    left=3cm,
    right=3cm,
    top=2.5cm,
    bottom=3cm
]{geometry}

\usepackage[shortlabels]{enumitem}
\usepackage{multicol}
\usepackage{pdflscape}
\usepackage[normalem]{ulem}
\usepackage{marginnote}
\usepackage{comment}
\usepackage{cancel}

\usepackage[dvipsnames]{xcolor}
\usepackage{graphicx}
\usepackage[export]{adjustbox}
\usepackage{subcaption}

\usepackage[all]{xy}

\usepackage{tikz}

\usetikzlibrary{
    arrows,
    arrows.meta,
    automata,
    backgrounds,
    calc,
    chains,
    decorations.markings,
    fit,
    matrix,
    mindmap,
    patterns,
    positioning,
    scopes,
    shapes,
    shapes.geometric
}

\usepackage{thmtools}
\usepackage{todonotes}

\setlist[enumerate]{
    itemsep=0.2em,
    topsep=0.25em
}

\makeatletter

\renewcommand{\@dotsep}{4.5}

\renewcommand{\l@section}{%
    \@tocline{1}{0.4em}{0em}{2.5em}{\bfseries}%
}

\renewcommand{\subsection}{%
  \@startsection{subsection}{2}%
  {\z@}%
  {.5\linespacing\@plus.7\linespacing}%
  {.4\linespacing}%
  {\normalfont\bfseries}%
}

\renewcommand{\l@subsubsection}{%
    \@tocline{3}{0.1em}{4.9em}{5.8em}{}%
}

\renewcommand{\tocsection}[3]{%
    \indentlabel{%
        \@ifnotempty{#2}{%
            \bfseries\ignorespaces#1 #2\quad
        }%
    }%
    \bfseries#3%
}

\renewcommand{\tocsubsection}[3]{%
    \indentlabel{%
        \@ifnotempty{#2}{%
            \ignorespaces#1 #2\quad
        }%
    }%
    #3%
}

\def\@tocline#1#2#3#4#5#6#7{%
    \relax
    \ifnum#1>\c@tocdepth
    \else
        \par
        \addpenalty\@secpenalty
        \addvspace{#2}%
        \begingroup
            \hyphenpenalty\@M
            \@ifempty{#4}{%
                \@tempdima
                \csname r@tocindent\number#1\endcsname
                \relax
            }{%
                \@tempdima#4\relax
            }%
            \parindent\z@
            \leftskip#3\relax
            \advance\leftskip\@tempdima\relax
            \rightskip\@pnumwidth plus1em
            \parfillskip-\@pnumwidth
            #5%
            \leavevmode
            \hskip-\@tempdima
            {#6}%
            \nobreak
            \leaders
            \hbox{%
                $\m@th
                \mkern\@dotsep mu
                \hbox{.}%
                \mkern\@dotsep mu$%
            }%
            \hfill
            \nobreak
            \hbox to\@pnumwidth{%
                \@tocpagenum{%
                    \ifnum#1=1
                        \bfseries
                    \fi
                    #7%
                }%
            }%
            \par
            \nobreak
        \endgroup
    \fi
}

\AtBeginDocument{%
    \expandafter\renewcommand
    \csname r@tocindent0\endcsname{0pt}%
}

\makeatother

\newtheorem{theorem}{Theorem}[section]
\newtheorem{introthm}{Theorem}

\newtheorem{lemma}[theorem]{Lemma}
\newtheorem{fact}[theorem]{Fact}

\theoremstyle{definition}

\theoremstyle{remark}

\newtheorem{remark}[theorem]{Remark}

\numberwithin{equation}{section}

\newcommand{\R}{\mathbb{R}}

\newcommand{\N}{\mathbb{N}}

\newcommand{\aco}{\operatorname{aco}}
\newcommand{\co}{\operatorname{co}}
\newcommand{\Exp}{\operatorname{exp}}
\newcommand{\Ext}{\operatorname{ext}}

\newcommand{\eps}{\varepsilon}

\newcommand{\polten}[1]{%
    \widetilde{\otimes}_{\pi_s}^{#1, \,s}%
}
\newcommand{\multiten}[1]{%
    \widetilde{\otimes}_{\pi}^{\,#1}%
}

\DeclareMathOperator{\id}{Id}
\DeclareMathOperator{\supp}{supp}

\renewcommand{\geq}{\geqslant}
\renewcommand{\leq}{\leqslant}

\newcommand{\pten}{\ensuremath{\widetilde{\otimes}_\pi}}

\definecolor{teal}{RGB}{0,128,128}
\definecolor{coral}{RGB}{220,90,70}

\usepackage{hyperref}
	\hypersetup{breaklinks=true,colorlinks=true,
linkcolor=coral,citecolor=coral,
urlcolor=MidnightBlue}

\usepackage{bookmark}

\begin{document}

\title[On Exposed (Symmetric) Tensors]{On Exposed (Symmetric) Tensors}
\dedicatory{ }

\author[S.~Dantas]{Sheldon Dantas}
\address[S.~Dantas]{Czech Technical University in Prague, FEE, Department of Mathematics, Technick\'a 2, 16627, Prague 6, Czech Republic\\
\href{https://orcid.org/0000-0001-8117-3760}{ORCID: \texttt{0000-0001-8117-3760}}}
\email{\texttt{sheldon.dantas@fel.cvut.cz}}
\urladdr{www.sheldondantas.com}

\author[M.~Mazzitelli]{Martin Mazzitelli}
\address[M.~Mazzitelli]{Instituto Balseiro, CNEA -- Universidad Nacional de Cuyo, CONICET, Argentina.}
\email{\texttt{martin.mazzitelli@ib.edu.ar}}

\author[J.~T.~Rodr\'iguez]{Jorge Tom\'as Rodr\'iguez}
\address[J.~T.~Rodr\'iguez]{NuCoMPA, Facultad de Cs. Exactas, Universidad Nacional del Centro de la Provincia de Buenos Aires, (7000) Tandil, Argentina and CONICET\\
\href{https://orcid.org/0000-0003-4693-2498}{ORCID: \texttt{0000-0003-4693-2498}}}
\email{\texttt{jtrodriguez@nucompa.exa.unicen.edu.ar}}

\begin{abstract}  We provide an example of a projective tensor product whose unit ball contains an exposed tensor which is not elementary. We then transfer the construction to the symmetric projective tensor product and obtain an analogous result in that setting. Finally, we extend both constructions to projective tensor products with an arbitrary number of factors and to symmetric projective tensor products of every degree.
\end{abstract}


\maketitle
\tableofcontents
 
\section{Introduction}

The study of extreme and exposed points of unit balls is a classical topic in Banach space theory. In the context of tensor products, a natural question is whether extremal properties force a tensor to be elementary. Several known results for projective tensor products point in this direction. Ruess and Stegall in \cite{RuessStegall1982} studied the extremal structure of duals of spaces of operators and, as a consequence of their results, one obtains that, under suitable Radon-Nikodým and approximation property assumptions, 
\begin{equation*}
    \Ext(B_{X^* \pten Y^*}) = \Ext(B_{X^*}) \otimes \Ext(B_{Y^*}).
\end{equation*} 
For stronger notions of extremality, Ruess and Stegall proved that every strongly exposed point of the unit ball of a projective tensor product is elementary \cite{RuessStegallExposed}. Werner obtained the analogous conclusion for denting points, in fact in the more general setting of closed bounded absolutely convex sets \cite{WernerDenting}. More recently, Garc\'{\i}a-Lirola, Grelier, Mart\'{\i}nez-Cervantes and Rueda Zoca studied preserved extreme points and weakly strongly exposed points in projective tensor products \cite{GarciaLirolaGrelierMartinezRueda}. In particular, they proved that, whenever the compact operators from $X$ into $Y^*$ separate the points of $X \pten Y$, every non-zero preserved extreme point of $B_{X \pten Y}$ is an elementary tensor. In the same paper, the authors explicitly asked whether every extreme point of $B_{X \pten Y}$ must be elementary. In a different direction, Aliaga et al. proved that every extreme point of a projective tensor product which is integral projective norm-attaining and witnessed by a Radon measure must be an elementary tensor \cite{AliagaDantasGuerreroJungRoldan}. Thus, although several stronger forms of extremality are known to force the tensors to be elementary, the corresponding question for arbitrary extreme points remained open.

A similar phenomenon appears in the symmetric setting. Ryan and Turett proved that, for every finite-dimensional real Banach space $X$, the exposed and extreme points of
$B_{\polten{N}{X}}$ coincide. Moreover, every such point is of the form $\pm x^{\otimes N}$ for some $x\in S_X$ \cite{RyanTurett}. The geometry of symmetric tensor products and spaces of homogeneous polynomials has also been studied by Boyd and Ryan \cite{BoydRyan}. Closely related results are known for integral homogeneous polynomials. Boyd and Lassalle obtained descriptions of their extreme and exposed points under suitable assumptions \cite{BoydLassalle}, while Dimant, Galicer and Garc\'{\i}a proved that, for every real Banach space $X$, the extreme points of the unit ball of the space of integral $N$-homogeneous polynomials are exactly tensors of the form $\pm \phi^N$ with $\phi \in S_{X^*}$ \cite{DimantGalicerGarcia}. Although these results concern a related dual setting, they give further evidence for the prevalence of elementary tensors and pure powers in extremal problems.

It is worth mentioning that there is also a useful analogy with Lipschitz-free spaces, where elementary molecules play a role similar to elementary tensors. The question of whether every extreme point of the unit ball of a Lipschitz-free space must be an elementary molecule was open for a long time and was recently solved affirmatively by Aliaga, Perneck\'a and Smith: for every complete metric space $M$, every extreme point of $B_{\mathcal{F}(M)}$ is an elementary molecule \cite{AliagaPerneckaSmith}.

The purpose of the present paper is to show that the analogous statement for tensor products is false in full generality. We construct a two-dimensional real Banach space $X$, an equivalent renorming $Y$ of $c_0$ and a tensor $u \in X \pten Y$ which is  rank two and is an exposed point of $B_{X \pten Y}$. We then transfer this construction to the symmetric setting and obtain a real Banach space $Z$ such that $B_{\polten{2}{Z}}$ contains an exposed point which is not of the form $\pm z^{\otimes 2}$. Finally, we extend these two constructions to projective tensor products with an arbitrary number of factors and to symmetric projective tensor products of every degree.

\section{Notation}

Throughout this manuscript, all the Banach spaces will be considered over the {\bf real} numbers. Let $X$ be a Banach space. We will denote by $B_X$ its unit ball and by $X^*$ its topological dual space. If $A \subseteq X$ is a subset of $X$, then we denote by $\co(A)$ and by $\aco(A)$ its convex hull and absolutely convex hull, respectively. Suppose $A \subseteq X$ is convex. A point $x \in A$ is said to be an extreme point of $A$ if, whenever $x = \frac{y+z}{2}$ with $y,z \in A$, we have that $y = z = x$. The set of all extreme points of $A$ is denoted by $\Ext(A)$. A point $ x \in A$ is said to be an exposed point of $A$ if there exists $f \in X^*$ such that $f(x) > f(y)$ for every $y \in A \setminus \{x\}$. In this case, we say that $f$ exposes $x$. The set of all exposed points of $A$ is denoted by $\Exp(A)$. A nonempty subset $F \subseteq A$ is said to be an exposed face of $A$ if there exists $f \in X^*$ such that $F= \{ x\in A: f(x) = \sup_{y \in A} f(y)\}$. In this case, we say that $f$ exposes the face $F$. Clearly, $\Exp(A) \subseteq \Ext(A)$ and an exposed point is precisely an exposed face consisting of a single point.

Let $X$ and $Y$ be Banach spaces. We denote by $X \otimes Y$ their algebraic tensor product. The projective norm on $X \otimes Y$ is defined by
\begin{equation*}
    \|u\|_{\pi} := \inf \left\{ \sum_{j=1}^n \|x_j\| \|y_j\| : u = \sum_{j=1}^n x_j \otimes y_j, \ n \in \N \right\}.
\end{equation*}
The completion of $X \otimes Y$ with respect to $\|\cdot\|_{\pi}$ is denoted by $X \pten Y$ and is called the projective tensor product of $X$ and $Y$. We recall that $(X \pten Y)^*$ is isometrically identified with $\mathcal{B}(X \times Y)$, the Banach space of all bounded bilinear forms $B: X \times Y \rightarrow \R$ through the duality $\langle B, x \otimes y \rangle = B(x,y)$ for every $x \in X$ and $y \in Y$. More generally, given Banach spaces $X_1,\ldots,X_N$, we denote by
\(\multiten{N} X_i
\)
the completion of the algebraic tensor product with respect to the projective norm
\[
\|u\|_\pi
:=
\inf\left\{
\sum_{j=1}^n
\|x_{1,j}\|\cdots\|x_{N,j}\| :
u=
\sum_{j=1}^n
x_{1,j}\otimes\cdots\otimes x_{N,j},
\ n\in\mathbb N
\right\}.
\]

For $N \in \N$, we denote by $\bigotimes^{N,s} X$ the algebraic symmetric tensor product of $X$, that is, the linear span of the elements $x^{\otimes N} = x \otimes \cdots \otimes x$ with $x \in X$. The symmetric projective norm on $\bigotimes^{N,s} X$ is given by
\begin{equation*}
    \|u\|_{\pi_s} := \inf \left\{ \sum_{j=1}^n |\lambda_j| \|x_j\|^N : u = \sum_{j=1}^n \lambda_j x_j^{\otimes N}, \ n \in \N \right\}.
\end{equation*}
Its completion is called the symmetric projective tensor product. For $w_1,\ldots,w_N   \in X$, we write 
\begin{equation} \label{sym}
    w_1 \vee \cdots \vee w_N := \frac{1}{N!}\sum_{\sigma \in S_N} w_{\sigma_1}\otimes\cdots\otimes w_{\sigma_N},
\end{equation}
where $S_N$ is the permutation group.

We denote by $\mathcal{L}_s(^N X)$ the space of symmetric $N$-linear forms and by $\mathcal{P}(^N X)$ the Banach space of all continuous $N$-homogeneous polynomials on $X$. Every $P \in \mathcal{P}(^N X)$ has a unique associated symmetric $N$-linear form $\check{P} \in \mathcal{L}_s(^N X)$ satisfying $P(x) = \check{P}(x,\ldots,x)$ for every $x \in X$. The dual space $(\polten{N}{X})^*$ is isometrically identified with $\mathcal{P}(^N X)$ through $\langle P, x^{\otimes N} \rangle = P(x)$ for every $x\in X$ and $\langle P, w_1 \vee \cdots \vee w_N \rangle = \check{P}(w_1,\ldots, w_N)$ for every $w_1,\ldots,w_N   \in X$.

\section{The projective tensor product }\label{sec:full tensor}

We start with the projective tensor product. The main goal of this section is to prove the existence of an equivalent renorming of $c_0$ such that the unit ball of $X \pten c_0$ contains a non-elementary exposed point, where $X$ is some two-dimensional space which we will specify later. More specifically, we will prove the following result.

\begin{introthm} \label{theorem:full-projective-main} There exist a two-dimensional real Banach space $X$, a renorming of $c_0$, denoted by $Y$, and  $u \in X \pten Y$ such that $u$ is a non-elementary exposed point of the unit ball $B_{X \pten Y}$.
\end{introthm}

As we have mentioned before, Ruess and Stegall characterized the extreme points of the dual ball of suitable spaces of compact weak$^*$-to-weak continuous operators (see \cite[Theorem~1.1]{RuessStegall1982}). Combined with the classical tensorial identification (see, for instance, \cite[Theorem~5.33]{Ryan}) this yields that, whenever $E^*$ or $F^*$ has the Radon-Nikodým property and $E^*$ or $F^*$ has the approximation property, then 
\begin{equation*}
    \Ext(B_{E^* \pten F^*}) = \Ext(B_{E^*}) \otimes \Ext(B_{F^*}).
\end{equation*}
In particular, every extreme point of $B_{\R^2 \pten \ell_{\infty}}$ must be an elementary tensor. This means that, if we would like to find a non-elementary extreme point in a tensor product of the form $\R^2 \pten c_0$, then such a point needs to fail to be extreme in the corresponding larger tensor product $\R^2 \pten \ell_{\infty}$.

Based on this, the idea of our construction is as follows. At first, we consider only algebraic tensors and no norms involved. In this scenario, we construct three tensors $u_1, u_2, u_3 \in \R^2 \pten \ell_{\infty}$ such that their convex hull, denoted here by $\mathbf{S} := \co \{u_1, u_2, u_3\}$, is such that $\mathbf{S} \cap (\R^2 \pten c_0) = \{u\}$, where $u$ is a non-elementary (in fact, rank 2) tensor. Having this in mind, we then introduce suitable norms on $\R^2$ and $c_0$ so that $\mathbf{S}$ becomes an exposed face of the unit ball of the corresponding tensor product $\R^2 \pten \ell_{\infty}$, which is exposed by a functional $\varphi$. Since the intersection of $\mathbf{S}$ with $\R^2 \pten c_0$ consists only of $u$, the restriction of the same functional $\varphi$ to $\R^2 \pten c_0$ exposes $u$.

Now we are ready to start. As it requires several steps, we split our construction into different lemmas and facts, and provide a proof for Theorem \ref{theorem:full-projective-main} at the end of the section. Let us notice that some of the results presented in this section will be used in Section \ref{section:symmetric}, where we will be considering the symmetric case.

Consider $x_1, x_2, x_3 \in \R^2$ such that $x_1 + x_2 + x_3 = 0$ and assume that $x_1$ and $x_2$ are linearly independent. On the other hand, in $\ell_{\infty}$, we consider the vectors
\begin{equation*}
    z_1 := \left( 1, - \frac{1}{2}, 1, 1, \ldots \right), \ z_2:= \left( - \frac{1}{2}, 1, 1, 1, \ldots, \right) \ \mbox{and} \ z_3 := \left( -\frac{1}{2}, -\frac{1}{2}, 1, 1, \ldots \right).
\end{equation*}
Notice that, whenever $i \not= j$, $z_i - z_j$ is an element of $c_0$ as their tails are composed of 1's. The reader should take into account that these elements will be used throughout the entire paper and it is important to remember their definitions in order to avoid confusion. In $\R^2 \otimes \ell_{\infty}$, we consider the elements 
\begin{equation*}
    u_j := x_j \otimes z_j \ \mbox{for} \ j=1,2,3.
\end{equation*}
Now, we define the element
\begin{equation} \label{vector-u}
u:= \frac{u_1 + u_2 + u_3}{3} \in \R^2 \otimes \ell_{\infty}.
\end{equation}
Let us notice that, since $x_3 = -x_1 - x_2$, it turns out that $u = \frac{1}{2} x_1 \otimes e_1 + \frac{1}{2} x_2 \otimes e_2$, that is, $u$ is a rank 2 tensor as $x_1$ and $x_2$ are linearly independent. We highlight the first property of these elements.

\begin{fact} \label{fact1} Set $\mathbf{S} := \co \{u_1, u_2, u_3\}$. Then,  $\mathbf{S} \cap (\R^2 \otimes c_0) = \{ u \}$.
\end{fact}

\begin{proof} Let us take $v \in \mathbf{S} \cap (\R^2 \otimes c_0)$. Then, we can write $v = \sum_{i=1}^3 t_i u_i$ with $t_i \geq 0$ for every $i=1,2,3$ and $\sum_{i=1}^3 t_i = 1$. Since $x_3 = -x_1 - x_2$, then
\begin{equation*}
v = x_1 \otimes (t_1 z_1 - t_3z_3) + x_2 \otimes (t_2 z_2 - t_3 z_3).
\end{equation*} 
Since $x_1$ and $x_2$ are linearly independent, we must have that both elements $t_1 z_1 - t_3 z_3$ and $t_2 z_2 - t_3 z_3$ are elements of $c_0$. From the way the $z_j$'s are defined, we have that $t_1 = t_3$ and $t_2 = t_3$. This shows that $t_1 = t_2 = t_3 = 1/3$.
\end{proof}

Next, we renorm $c_0$ so that each $z_j$ lies in the unit ball of its bidual. Let $P_n: \ell_{\infty} \rightarrow \ell_{\infty}$ denote the projection onto the first $n$ coordinates. Consider a sequence $(r_n) \subseteq \R$ with $0 < r_n < 1$ for every $n \in \N$ and $r_n \uparrow 1$. For every $n \geq 3$, let us define 
\begin{equation*}
    w_{j,n} := r_n P_n(z_j) \ \mbox{for} \ j=1,2,3.
\end{equation*}
Then, $\|w_{j,n}\|_{\infty} = r_n < 1$ for every $n \geq 3$ and $j=1,2,3$. Notice that, for every $\xi = (\xi_k) \in \ell_1$, we have that 
\begin{equation*}
    \langle w_{j,n}, \xi \rangle = r_n \sum_{k=1}^n z_j(k) \xi_k \longrightarrow \sum_{k=1}^{\infty} z_j(k) \xi_k = \langle z_j, \xi \rangle 
\end{equation*}
for every $j=1,2,3$ as $n \rightarrow \infty$. This means that $(w_{j,n})_{n=3}^{\infty}$ converges weak-star to $z_j$ for every $j=1,2,3$. Now, consider $Y$ to be $c_0$ equipped with an equivalent norm defined by the unit ball
\begin{equation*}
    B_Y := \overline{\aco}^{\|\cdot\|_{\infty}} \left( \frac{1}{4} B_{c_0} \cup \{ w_{j,n}: j=1,2,3, n \geq 3 \} \right).
\end{equation*}
Notice that $B_Y$ is the smallest norm-closed convex and balanced subset of $c_0$ containing $\frac{1}{4}B_{c_0}$ and all the vectors $w_{j,n}$. More specifically, the new norm is given by the Minkowski functional of $B_Y$ and clearly 
\begin{equation*}
    \frac{1}{4} B_{c_0} \subseteq B_Y \subseteq B_{c_0}.
\end{equation*}
In particular, $Y^* = \ell_1$ and $Y^{**} = \ell_{\infty}$ as vector spaces. Let us consider the set 
\begin{equation*}
    K:= \frac{1}{4} B_{\ell_{\infty}} \cup \left\{ \pm w_{j,n}: j=1,2,3 \ \mbox{and} \ n \geq 3 \right\} \cup \left\{ \pm z_1, \pm z_2, \pm z_3 \right\}.
\end{equation*}
Since $(w_{j,n})_{n=3}^{\infty}$ converges weak-star to $z_j$ for every $j=1,2,3$ and $B_{c_0}$ is weak-star dense in $B_{\ell_{\infty}}$, then $K \subseteq B_{Y^{**}}$. Since $B_{Y^{**}}$ is convex and weak-star closed, we have that 
\begin{equation*}
    \overline{\co}^{w^*}(K) \subseteq B_{Y^{**}}.
\end{equation*}
On the other hand, $B_Y \subseteq \overline{\co}^{w^*}(K)$ and, by Goldstine's theorem, $B_{Y^{**}} = \overline{B_Y}^{w^*}$. Then, $B_{Y^{**}} \subseteq \overline{\co}^{w^*}(K)$. Therefore, we obtain
\begin{equation} \label{eq1}
    B_{Y^{**}} = \overline{\co}^{w^*}(K).
\end{equation}

In $\R^2$, we choose the points $x_1 := e_1 = (1,0)$, $x_2:= e_2 = (0,1)$ and $x_3 := -e_1 - e_2 = (-1,-1)$. We equip $\R^2$ with the norm whose unit ball is $\aco \{x_1, x_2, x_3\} = \co \{ \pm x_1, \pm x_2, \pm x_3\}$. Let us call this space $X$. The unit ball is therefore the hexagon with vertices $(1, 0)$, $(0,1)$, $(-1,-1)$, $(-1,0)$, $(0,-1)$ and $(1,1)$ (see Figure \ref{fig1} below).

\begin{figure}[ht]
\centering
\begin{tikzpicture}[scale=2]

\fill[gray!15]
    (1,0) --
    (1,1) --
    (0,1) --
    (-1,0) --
    (-1,-1) --
    (0,-1) -- cycle;

\draw[thick]
    (1,0) --
    (1,1) --
    (0,1) --
    (-1,0) --
    (-1,-1) --
    (0,-1) -- cycle;

\draw (-1.4,0) -- (-1,0);
\draw[dotted] (-1,0) -- (1,0);
\draw[->] (1,0) -- (1.5,0);

\draw (0,-1.4) -- (0,-1);
\draw[dotted] (0,-1) -- (0,1);
\draw[->] (0,1) -- (0,1.5);

\fill (1,0) circle (1.2pt);
\fill (1,1) circle (1.2pt);
\fill (0,1) circle (1.2pt);
\fill (-1,0) circle (1.2pt);
\fill (-1,-1) circle (1.2pt);
\fill (0,-1) circle (1.2pt);

\node[below right] at (1,0) {$(1,0)$};
\node[above right] at (1,1) {$(1,1)$};
\node[above left] at (0,1) {$(0,1)$};

\node[above left] at (-1,0) {$(-1,0)$};
\node[below left] at (-1,-1) {$(-1,-1)$};
\node[below right] at (0,-1) {$(0,-1)$};

\node at (0.35,0.25) {$B_X$};

\end{tikzpicture}
\caption{The unit ball
$B_X=\operatorname{aco}\{x_1,x_2,x_3\}$}
\label{fig1}
\end{figure}

Let us notice that this norm comes from the quotient $\ell_1^3 / \langle (1,1,1) \rangle$. This will be convenient for us when proving Fact \ref{fact4} as we will take advantage of the fact that projective tensor norms preserve quotient maps (see, for instance, \cite[Proposition 2.5]{Ryan}). Indeed, let $q: \ell_1^3 \rightarrow \R^2$ be given by $q(e_1):= x_1$, $q(e_2) := x_2$ and $q(e_3):= x_3$. Then, $q(a,b,c) = (a-c, b-c)$ for every $(a,b,c) \in \ell_1^3$ and $q(a,b,c) = 0$ if and only if $a=c$ and $b=c$. This shows that $\ker q = \langle (1,1,1) \rangle$ and then $X$ is isomorphic to $\ell_1^3 / \langle (1,1,1) \rangle$. Since $B_{\ell_1^3} = \aco\{e_1, e_2, e_3\}$, we have that $q(B_{\ell_1^3}) = \aco \{ q(e_1), q(e_2), q(e_3)\} = B_X$. In particular, 
\begin{equation*}
    \|x\|_X = \inf \{ |a| + |b| + |c|: x = a x_1 + b x_2 + c x_3 \}.
\end{equation*}

We need the following lemma.

\begin{lemma} \label{lemma1} Let $Y = (c_0, \|\cdot\|_Y)$ as defined above. For every $n \in \N$, consider the coordinate functionals $e_n^* \in Y^*$. Let us set $f_1 := e_1^*$, $f_2 := e_2^*$ and $f_3:= -e_1^* - e_2^*$. Then, $\|f_j\|_{Y^*} = 1$ for every $j=1,2,3$ and 
\begin{equation} \label{eq2}
    \{ z \in B_{Y^{**}}: z(f_j) = 1 \} = \{z_j\}. 
\end{equation}
\end{lemma}

\begin{proof} Suppose first that $j=1,2$. If $z \in \frac{1}{4}B_{\ell_{\infty}}$, then $|f_i(z)| \leq 1/4$. Also,  $|f_i(w_{j,n})| \leq r_n < 1$ for every $n \geq 3$ and $f_i(z_j) = 1$ if and only if $i=j$. So, for $i=1,2$, we have that 
\begin{equation} \label{lemmaeq1}
    \{z \in K: f_i(z) = 1 \} = \{z_i\}.
\end{equation}
 Now, let us consider the case where $j=3$. If $z \in \frac{1}{4} B_{\ell_{\infty}}$, then $|f_3(z)| \leq 1/2$. Also, $|f_3(w_{1,n})| = |f_3(w_{2,n})|= \frac{r_n}{2}$ while $|f_3(w_{3,n})|=r_n < 1$ for every $n \geq 3$. So,
 \begin{equation} \label{lemmaeq2}
 \{ z \in K: f_3(z) = 1\} = \{z_3\}.
 \end{equation}
Therefore, (\ref{lemmaeq1}) and (\ref{lemmaeq2}) show that $\{z \in K: f_j(z)=1\} = \{z_j\}$ for every $j=1,2,3$. We want to see that the same holds true for the unit ball $B_{Y^{**}}$ which by (\ref{eq1}) is equal to $\overline{\co}^{w^*}(K)$. Fix $j \in \{1,2,3\}$ and take $z \in B_{Y^{**}}$ with $f_j(z) = 1$. Since $B_{Y^{**}} = \overline{\co}^{w^*}(K)$, there exists a net $(z_{\alpha}) \subseteq \co(K)$ such that $z_{\alpha} \stackrel{w^*}{\longrightarrow} z$. We can write 
\begin{equation*}
    z_{\alpha} = \sum_{k=1}^{m_{\alpha}} \lambda_{\alpha, k} z_{\alpha, k}
\end{equation*}
where $z_{\alpha, k} \in K$, $\lambda_{\alpha, k} \geq 0$ and $\sum_{k=1}^{m_{\alpha}} \lambda_{\alpha, k} = 1$ for some $m_{\alpha} \in \N$. Then $f_j(z_{\alpha}) \longrightarrow f_j(z)$ for every $j=1,2,3$. Let us fix $j \in \{1,2,3\}$ and let $U$ be any $w^*$-neighborhood of $z_j$. Since $K \setminus U$ is compact and $z_j$ is the unique point of $K$ at which $f_j(z_j) = 1$, we have that $\sup_{x \in K \setminus U} f_j(x) < 1$. Therefore, there exists $\delta > 0$ such that $f_j(x) \leq 1 - \delta$ for every $x \in K \setminus U$. Consider now 
\begin{equation*}
    a_{\alpha} := \sum_{\{k: z_{\alpha, k} \not\in U\}} \lambda_{\alpha, k}.
\end{equation*}
Then,
\begin{equation*}
    f_j(z_{\alpha}) = \sum_{\{k: z_{\alpha, k} \in U\}} \lambda_{\alpha, k} f_j(z_{\alpha, k}) + \sum_{\{k: z_{\alpha, k} \not\in U\}} \lambda_{\alpha, k} f_j(z_{\alpha, k}) 
    \leq 1 - \delta a_{\alpha}.
\end{equation*}
Since $f_j(z_{\alpha}) \longrightarrow 1$, we have that $a_{\alpha} \longrightarrow 0$. This means that, asymptotically, the convex combinations $z_{\alpha}$ lie in every $w^*$-neighborhood $U$ of $z_j$, that is, $z_{\alpha} \stackrel{w^*}{\longrightarrow} z_j$. Therefore, $z = z_j$ and we are done. 
\end{proof}

Retain all the notation we have considered so far, including the one from Lemma \ref{lemma1}. Now, let us define the bounded linear operator $A: X \rightarrow Y^*$ given by $A(x_1) := f_1$ and $A(x_2) := f_2$. In particular,  $A(x_3) = A(-x_1 -x_2) = -f_1 - f_2 = f_3$. Clearly, we have that $\|A\| = 1$. Since we are working with a very specific shape of the unit ball (see again Figure \ref{fig1}), we have the following straightforward fact. We present a proof for the sake of completeness.

\begin{fact} \label{fact3} The bounded linear operator $A: X \rightarrow Y^*$ defined above attains its norm only at $\pm x_j$ for every $j=1,2,3$.
\end{fact}

\begin{proof} Let $x \in S_X$ be different from $\pm x_j$ for every $j=1,2,3$. Then, $x$ belongs to the relative interior of one of the sides of the hexagon (see Figure \ref{fig1}), which means we can write $x = t \eps_i x_i + (1 - t) \eps_j x_j$, where $0 < t < 1$, $\eps_i, \eps_j \in \{-1,1\}$ and $\eps_i x_i \not= \eps_j x_j$. Suppose that $\|Ax\|_{Y^*} = 1$. Then, there exists $z \in B_{Y^{**}}$ such that $z(Ax)=1$, that is,
\begin{equation*}
    1 = t \eps_i z(f_i) + (1 - t) \eps_j z(f_j) \leq 1
\end{equation*}
which implies that $\eps_i z(f_i) = \eps_j z(f_j) = 1$. By Lemma \ref{lemma1}, we must have $\eps_i z = z_i$ and $\eps_j z = z_j$, that is, $z = \eps_i z_i = \eps_j z_j$ as $\eps_i^2 = \eps_j^2 = 1$. This is impossible for two different vertices.
\end{proof}

\begin{fact} \label{fact4} Let $Z$ be any Banach space. Let $v \in X \pten Z$. Then,
\begin{equation*}
    \|v\|_{\pi} = \inf \left\{ \|w_1\| + \|w_2\| + \|w_3\|: v = \sum_{j=1}^3 x_j \otimes w_j \right\}.
\end{equation*}
\end{fact}

\begin{proof} Consider the quotient map $q: \ell_1^3 \rightarrow X$ as before. Given that $X = \ell_1^3 / \langle (1,1,1) \rangle$, we have that $q$ is a quotient map. Since the projective tensor product preserves quotients, $q \otimes \id_Z: \ell_1^3 \pten Z \rightarrow X \pten Z$ is a quotient map. Now, $\ell_1^3 \pten Z$ is isometrically equal to $Z \oplus_1 Z \oplus_1 Z$, that is, every element of $\ell_1^3 \pten Z$ can be written uniquely as $\sum_{j=1}^3 e_j \otimes w_j$, for some $w_j \in Z$ for $j=1,2,3$ and its norm is 
\begin{equation*}
    \left\| \sum_{j=1}^3 e_j \otimes w_j \right\|_{\pi} = \|w_1\| + \|w_2\| + \|w_3\|.
\end{equation*}
Applying $q \otimes \id_Z$, we get 
\begin{equation*}
    (q \otimes \id_Z) \left( \sum_{j=1}^3 e_j \otimes w_j \right) = \sum_{j=1}^3 x_j \otimes w_j.
\end{equation*}
Since $q \otimes \id_Z$ is a quotient map, the norm of $v$ is the quotient norm, that is, $\|v\|_{\pi} = \inf \{ \|w\|_{\pi}: (q \otimes \id_Z)(w) = v \}$. Since every $w$ has the form $w = \sum_{j=1}^3 e_j \otimes w_j$ with $v = \sum_{j=1}^3 x_j \otimes w_j$ and $\|w\|_{\pi} = \|w_1\| + \|w_2\| + \|w_3\|$, we are done.
\end{proof}

\begin{remark} \label{remark1} Let us observe that, in Fact \ref{fact4}, if $Z$ is a dual space, the infimum is always attained. Indeed, in this case we can use the $w^*$-compactness of the unit balls. Therefore, a minimizing net admits a $w^*$-convergent subnet whose limit attains the infimum.
\end{remark}

We need one more result before proving Theorem \ref{theorem:full-projective-main}.

\begin{lemma} \label{lemma2} The convex hull $\mathbf{S} = \co \{u_1, u_2, u_3\}$ is an exposed face of $B_{X \pten Y^{**}}$.
\end{lemma}

\begin{proof} Let us define $\varphi \in (X \pten Y^{**})^*$ by $\varphi(x \otimes z) := z(Ax)$ 
for every $x \in X$ and $z \in Y^{**}$. Since $\|A\| = 1$, it follows that $\|\varphi\| \leq 1$. We will see that $\varphi$ exposes $\mathbf{S}$. For every $j=1,2,3$, we have that
\begin{equation*}
    \varphi(x_j \otimes z_j) = z_j(Ax_j) = z_j(f_j) = 1.
\end{equation*}
Since $\|x_j\|_X = \|z_j\|_{Y^{**}} = 1$, then $u_j = x_j \otimes z_j \in B_{X \pten Y^{**}}$. This shows that $\mathbf{S} \subseteq \{ v \in B_{X \pten Y^{**}}: \varphi(v) = 1\}$. For the other inclusion, let $v \in B_{X \pten Y^{**}}$ with $\varphi(v) = 1$. Since $\varphi(v) = 1$ and $\|\varphi\| \leq 1$, we have that $\|v\|_{\pi} = 1$. By Fact \ref{fact4} and Remark \ref{remark1}, there exist $w_1, w_2, w_3 \in Y^{**}$ such that 
\begin{equation*}
    v = \sum_{j=1}^3 x_j \otimes w_j \ \ \ \mbox{with} \ \ \ \sum_{j=1}^3 \|w_j\| = \|v\|_{\pi} = 1.
\end{equation*}
So,
\begin{equation*}
    1 = \varphi(v) = \sum_{j=1}^3 w_j(f_j) \leq \sum_{j=1}^3 \|w_j\| = 1.
\end{equation*}
This shows that $w_j(f_j) = \|w_j\|$ for every $j=1,2,3$. By Lemma \ref{lemma1}, we have that $w_j = \|w_j\| z_j$ for every $j=1,2,3$ with $\sum_{j=1}^3 \|w_j\| = 1$. This means that 
\begin{equation*}
    v = \sum_{j=1}^3 \|w_j\| x_j \otimes z_j = \sum_{j=1}^3 \|w_j\| u_j \in \mathbf{S}.
\end{equation*}
\end{proof}

Finally we can prove the main result of this section.

\begin{proof}[Proof of Theorem \ref{theorem:full-projective-main}] Consider the spaces $X$ and $Y$ constructed above and let $u_j = x_j \otimes z_j$ for every $j=1,2,3$, where $z_1 - z_3=\frac{3}{2}e_1$ and $z_2 - z_3=\frac{3}{2}e_2$. Consider also $u$ as in (\ref{vector-u}). Since $x_1+x_2 + x_3 = 0$, we have that 
\begin{equation*}
    u = \frac{1}{2} \big( x_1 \otimes e_1 + x_2 \otimes e_2\big ) \in X \pten Y
\end{equation*}
and it is rank 2 as $x_1$ and $x_2$ are linearly independent. Lemma \ref{lemma2} says that $\mathbf{S} =\co \{ u_1, u_2, u_3 \}$ is an exposed face of $B_{X \pten Y^{**}}$ and Fact \ref{fact1} says that $\mathbf{S} \cap (X \pten Y) = \{u\}$. Therefore, the restriction to $X \pten Y$ of the functional $\varphi$ which exposed this face exposes $u$. In other words, $u \in \Exp(B_{X \pten Y})$ and we are done.  
\end{proof}

\section{The symmetric projective tensor product case} \label{section:symmetric}

The main result of this section is the symmetric counterpart of Theorem \ref{theorem:full-projective-main}.

\begin{introthm}\label{theorem:symmetric-main} There exist a real Banach space $Z$ and an exposed point of the unit ball $B_{\polten{2}{Z}}$ which is not of the form $\pm w^{\otimes 2}$ for any $w \in Z$.
\end{introthm}

The main idea to prove Theorem \ref{theorem:symmetric-main} is to transfer the exposed tensor obtained in Section \ref{sec:full tensor} to a symmetric projective tensor product. More precisely, we will construct a Banach space $Z$ such that $X \pten Y$ can be identified through an isometry $J: X \pten Y \rightarrow \polten{2}{Z}$ with a 1-complemented subspace of $\polten{2}{Z}$. We denote by $M: \polten{2}{Z} \rightarrow J(X \pten Y)$ the corresponding contractive projection and we set $u_s := J(u)$. If $v \in B_{\polten{2}{Z}}$ satisfies $\Psi(v) = 1$, where $\Psi := \varphi \circ J^{-1} \circ M$, then, since both $M$ and $J^{-1}$ are contractive on their respective ranges, $\|J^{-1} M(v)\|_{\pi} \leq 1$. Since $\varphi$ exposes $u$ from the previous section, it follows that $J^{-1} M(v) = u$ and hence $M(v) = u_s$. Therefore, in order to prove that $\Psi$ exposes $u_s$, it will be enough to show that 
\begin{equation*}
B_{\polten{2}{Z}} \cap M^{-1}(u_s) = \{ u_s \}. 
\end{equation*}
That is precisely what we do in the remainder of this section. We ask the reader to keep in mind the notation introduced in the previous section, which will be used throughout the passage to the symmetric setting. The decomposition of symmetric tensor products of direct sums is done in \cite[Sections 2.2 and 3.4]{AnsemilFloret}. It provides a useful tool for transferring examples from the projective tensor product to the symmetric projective tensor product (see also \cite[Sections~2.8 and~2.9]{FloretSymmetric}).

We start with the transferring result, which is well-known to experts (see Fact \ref{symmetric:fact1} below). Nevertheless, we present a proof for completeness. Let $X$ and $Y$ be the Banach spaces constructed in Section \ref{sec:full tensor}. The Banach space of Theorem \ref{theorem:symmetric-main} is the following direct sum
\begin{equation*}
    Z := (X \oplus_{\infty} Y, \| \cdot \|_{\infty}).
\end{equation*}

\begin{fact} \label{symmetric:fact1} The map $J: X \otimes Y \rightarrow \bigotimes^{2,s} Z$ given by 
\begin{equation*}
    J(x \otimes y) := 2 (x,0) \vee (0, y) 
\end{equation*}
for every $x \in X$ and $y \in Y$, extends to a linear isometry $J: X \pten Y \rightarrow \polten{2}{Z}$ and its range is 1-complemented.
\end{fact}

\begin{proof} Recall the definition of $\vee$ in (\ref{sym}). For arbitrary elements $v,w$ we have the identity $(v+w)^{\otimes 2} - (v-w)^{\otimes 2} = 4 v \vee w$, which implies that 
\begin{equation} \label{symmetric-eq1}
    2 v \vee w = \frac{1}{2} \big( (v+w)^{\otimes 2} - (v-w)^{\otimes 2} \big).
\end{equation}
Now, take $v = \left( \sqrt{ \frac{\|y\|}{\|x\|}} x, 0 \right)$ and $w = \left( 0, \sqrt{ \frac{\|x\|}{\|y\|}} y \right)$ for non-zero $x \in X$ and $y \in Y$. Then, we have that 
\begin{equation*}
    v \vee w = \left( \sqrt{ \frac{\|y\|}{\|x\|}} x, 0 \right) \vee \left( 0, \sqrt{ \frac{\|x\|}{\|y\|}} y \right) = (x,0) \vee (0,y).
\end{equation*}
Hence, $J(x \otimes y) = 2 v \vee w$ and by (\ref{symmetric-eq1}) we get 
\begin{equation*}
    J(x \otimes y) = \frac{1}{2} \left[ \left( \sqrt{ \frac{\|y\|}{\|x\|}} x, \sqrt{ \frac{\|x\|}{\|y\|}} y \right)^{\otimes 2} - \left( \sqrt{ \frac{\|y\|}{\|x\|}} x, -\sqrt{ \frac{\|x\|}{\|y\|}} y \right)^{\otimes 2}\right].
\end{equation*}
Since both vectors inside the squares have norm $\sqrt{\|x\| \|y\|}$ in $Z$, we obtain 
\begin{equation*}
\|J(x \otimes y)\|_{\pi_s} \leq \|x\| \|y\|.
\end{equation*}
This implies that $\|Jw\|_{\pi_s} \leq \|w\|_{\pi}$ for every $w \in X \otimes Y$. Conversely, let $B \in \mathcal{B}(X \times Y)$ be a bounded bilinear form and define $Q_B(x,y) := B(x,y)$ for every $(x, y) \in Z$. Then, $Q_B \in \mathcal{P}(^2 Z)$ and $\|Q_B\| = \|B\|$. Moreover, the symmetric form $\check{Q}_B \in \mathcal{L}_s(^2 Z)$ associated to $Q_B$ is given by 
\begin{equation*}
\check{Q}_B((x_1, y_1), (x_2, y_2)) = \frac{1}{2} (B(x_1, y_2) + B(x_2, y_1))
\end{equation*}
for every $(x_1, y_1), (x_2, y_2) \in Z$. By duality 
\begin{equation*}
    \langle Q_B, J(x \otimes y) \rangle = \langle Q_B, 2(x,0) \vee (0,y) \rangle  = 2 \check{Q}_B((x,0), (0,y)) = B(x,y)
\end{equation*}
for every $x \in X$ and $y \in Y$. Thus, for every $w \in X \otimes Y$, we have that $| \langle B, w \rangle| \leq \|B\| \|Jw\|_{\pi_s}$. Taking the supremum over $B \in B_{\mathcal{B}(X \times Y)}$ gives $\|w\|_{\pi} \leq \|Jw\|_{\pi_s}$. Therefore, $J$ extends to an isometry.

Now, define $S: Z \rightarrow Z$ by $S(x,y) := (x,-y)$ for every $(x,y) \in Z$. Clearly, $S$ is a surjective isometry and $S^2 = \id_Z$. Let now $S^{\otimes_{2,s}}: \polten{2}{Z} \rightarrow \polten{2}{Z}$ be the isometry induced by $S$. Thus, $S^{\otimes_{2,s}}(v \vee w) = Sv \vee Sw$ for every $v,w \in Z$. In particular,
\begin{equation*}
    S^{\otimes_{2,s}}((x,y)^{\otimes 2}) = (x,-y)^{\otimes2}
\end{equation*}
for every $x \in X$ and $y \in Y$. Define
\begin{equation*}
    M:= \frac{\id_{\polten{2}{Z}} - S^{\otimes_{2,s}}}{2}.
\end{equation*}
Then, $\|M\| \leq 1$ and since $(S^{\otimes_{2,s}})^2 = \id_{\polten{2}{Z}}$, we have that 
\begin{equation*}
    M^2 = \frac{1}{4}(\id_{\polten{2}{Z}} - S^{\otimes_{2,s}})^2 = \frac{1}{4} (\id_{\polten{2}{Z}} - 2 S^{\otimes_{2,s}} + (S^{\otimes_{2,s}})^2) = M.
\end{equation*}
So, $M$ is a projection and, in particular, $\|M\| \geq 1$. We now identify its range. For every $x \in X$ and $y \in Y$, we have that 
\begin{equation*}
    M((x,y)^{\otimes 2}) = \frac{1}{2} \left[ (x,y)^{\otimes 2} - (x,-y)^{\otimes 2} \right].
\end{equation*}
Using that $(x,y) = (x,0) + (0,y)$ and $(x,-y) = (x,0) - (0,y)$, we obtain that $(x,y)^{\otimes 2} = (x,0)^{\otimes 2} + 2(x,0) \vee (0,y) + (0,y)^{\otimes 2}$ and $(x,-y)^{\otimes 2} = (x,0)^{\otimes 2} - 2(x,0) \vee (0,y) + (0,y)^{\otimes 2}$. Therefore,
\begin{equation} \label{symmetric-eq2}
    M((x,y)^{\otimes 2}) = 2(x,0) \vee (0,y) = J(x \otimes y).
\end{equation}
Since the algebraic symmetric tensor product is spanned by tensors of the form $v^{\otimes 2}$, we get that $M(\polten{2}{Z}) \subseteq J(X \pten Y)$ and (\ref{symmetric-eq2}) gives $M(\polten{2}{Z}) = J(X \pten Y)$.
\end{proof}

Define $u_s := J(u)$. Since $J$ is an isometry, we get that 
\begin{equation*}
    \|u_s\|_{\pi_s} = \|Ju\|_{\pi_s} = \|u\|_{\pi} = 1.
\end{equation*}
Now, let us define 
\begin{equation*}
\Psi := \varphi \circ J^{-1} \circ M 
\end{equation*}
where $\varphi$ is defined in the proof of Lemma \ref{lemma2}. Observe that $\Psi: \polten{2}{Z} \rightarrow \R$. Moreover, since $u_s$ belongs to the range of $M$, it follows that $M(u_s) = u_s$ and, therefore,
\begin{equation*}
    \Psi(u_s) = \varphi(J^{-1}(M(u_s))) = \varphi(J^{-1}(Ju)) = \varphi(u) = 1.
\end{equation*}
Furthermore, we have that 
\begin{equation*} 
    \{ v \in B_{\polten{2}{Z}}: \Psi(v) = 1 \} = B_{\polten{2}{Z}} \cap M^{-1}(u_s).
\end{equation*}
Indeed, recall that $\{w \in B_{X \pten Y}: \varphi(w) = 1\} = \{u\}$. Take $v \in B_{\polten{2}{Z}}$ such that $\Psi(v) = 1$. By definition, $\varphi(J^{-1}(Mv)) = 1$. Since $M$ is contractive and $J^{-1}$ is an isometry on the range of $M$, then 
\begin{equation*}
    \|J^{-1}(Mv)\|_{\pi} = \|Mv\|_{\pi_s} \leq \|v\|_{\pi_s} \leq 1.
\end{equation*}
This means that $J^{-1} (Mv) \in B_{X \pten Y}$. Together with $\varphi(J^{-1}(Mv)) = 1$, this shows that $J^{-1}(Mv) = u$. Therefore, $Mv = Ju = u_s$ and $v \in M^{-1}(u_s)$. This proves the first inclusion. Now, take $v \in B_{\polten{2}{Z}} \cap M^{-1}(u_s)$. Then, $Mv = u_s$ and so 
\begin{equation*}
\Psi(v) = \varphi(J^{-1}(Mv)) = \varphi(J^{-1}(u_s)) = \varphi(J^{-1}(Ju)) = \varphi(u) = 1.
\end{equation*}

We need one last result before proving Theorem \ref{theorem:symmetric-main}.

\begin{lemma} \label{symmetric1} For every $v \in B_{\polten{2}{Z}}$ with $Mv = u_s$, we have $v = u_s$. In particular,
\begin{equation*} 
    \{ v \in B_{\polten{2}{Z}}: \Psi(v) = 1 \} = \{u_s\}.
\end{equation*}
\end{lemma}

\begin{proof} Let $P_X(x,y) := (x,0)$ and $P_Y(x,y) := (0,y)$ for every $x \in X$ and $y \in Y$ be the coordinate projections. They induce operators on $\polten{2}{Z}$ which we denote by the same symbols: $P_X(z^{\otimes 2}) = (P_X z)^{\otimes 2}$ and $P_Y(z^{\otimes 2}) = (P_Yz)^{\otimes 2}$ for every $z \in Z$. In particular, if $z = (x,y) \in Z$, we have that 
\begin{equation*}
    P_X((x,y)^{\otimes 2}) = (x,0)^{\otimes 2} \ \ \ \mbox{and} \ \ \ P_Y((x,y)^{\otimes 2}) = (0,y)^{\otimes 2}.
\end{equation*}
We have proved in (\ref{symmetric-eq2}) that $M((x,y)^{\otimes 2}) = 2(x,0) \vee (0,y)$ for every $x \in X$ and $y \in Y$. So, 
\begin{equation*}
    (M + P_X + P_Y)((x,y)^{\otimes 2}) = (x,y)^{\otimes 2}
\end{equation*}
for every $(x,y) \in Z$. This shows that 
\begin{equation*}
\id_{\polten{2}{Z}} = M + P_X + P_Y. 
\end{equation*}
In particular, $v = u_s + P_X(v) + P_Y(v)$. All we need to do now is to prove that $P_X(v) = P_Y(v) = 0$ and that is what we do in the rest of the proof.

Since $M$ is contractive and $Mv = u_s$, we have that 
\begin{equation*}
    1 = \|u_s\|_{\pi_s} = \|Mv\|_{\pi_s} \leq \|v\|_{\pi_s} \leq 1
\end{equation*}
and then $\|v\|_{\pi_s} = 1$. So, 
\begin{equation*}
    \Psi(v) = \varphi(J^{-1}(Mv)) = \varphi(u) = 1.
\end{equation*}

Now consider $K:= \{-1,1\} \times B_X \times B_{Y^{**}}$, where $B_X$ has the norm topology and $B_{Y^{**}}$ the $w^*$-topology. Notice that $K$ is a compact metrizable set since $X$ is finite-dimensional and $Y^* = \ell_1$ is separable (and, therefore, $B_{Y^{**}}$ is $w^*$-metrizable). We need the following claim, which uses the bounded linear operator $A: X \rightarrow Y^*$ defined just before Fact \ref{fact3}.

\vspace{0.2cm}
\noindent
{\it Claim}: $\eps z(Ax) = 1$ for $(\eps, x, z) \in K$ if and only if $(\eps, x, z) = (\eps, s x_j, \eps s z_j)$ for some $j \in \{1,2,3\}$ and $s, \eps \in \{-1,1\}$.

Suppose first that $\eps z(Ax) = 1$ for some $(\eps, x, z) \in K$. Since $\eps \in \{-1,1\}$, $|z(Ax)| = 1$. As $z \in B_{Y^{**}}$, $x \in B_X$ and $\|A\| = 1$, we have that $\|Ax\| = 1$. By Fact \ref{fact3}, $A$ attains its norm only at $\pm x_1, \pm x_2, \pm x_3$. Then, there are $j \in \{1,2,3\}$, $s \in \{-1,1\}$ such that $x = s x_j$. Since $A(x_j) = f_j$, we obtain $Ax = s f_j$. Thus, $1 = \eps z(Ax) = \eps s z(f_j)$. Therefore, $(\eps s z)(f_j) = 1$. Since $\eps s z \in B_{Y^{**}}$ and $\{w \in B_{Y^{**}}: w(f_j) = 1\} = \{z_j\}$, we get that $\eps s z = z_j$, which means that $z = \eps s z_j$. Conversely, suppose that $x = s x_j$ and $z = \eps s z_j$ for some $j \in \{1,2,3\}$ and $s, \eps \in \{-1,1\}$. Then, $Ax = s f_j$ and therefore $\eps z(Ax) = \eps (\eps s z_j)(s f_j) = z_j(f_j) = 1$.

Let now $\mathcal{E}:= \{(\eps, s x_j, \eps s z_j): j=1,2,3, \ \mbox{with} \ s, \eps \in \{-1,1\} \}$. By the above claim, we have that 
\begin{equation*}
    \mathcal{E} = \{ (\eps, x, z) \in K: \eps z(Ax) = 1\} 
\end{equation*}
which is a finite set. Choose pairwise disjoint neighborhoods $U_{j, \eps, s}^n$ of $(\eps, s x_j, \eps s z_j)$ whose diameters tend to zero as $n \rightarrow \infty$ and set 
\begin{equation*}
    U_n := \bigcup_{j, \eps, s} U_{j, \eps, s}^n.
\end{equation*}
Since $(\eps, x, z) \mapsto \eps z(Ax)$ is continuous on the compact space $K$, the claim above gives 
\begin{equation*} 
    \max_{(\eps,x,z) \in K \setminus U_n} \eps z(Ax) < 1. 
\end{equation*}
For every $n \in \N$, choose a representation 
\begin{equation*}
    v = \sum_{k=1}^{\infty} c_{n,k} \eps_{n,k} (x_{n,k}, y_{n,k})^{\otimes 2}
\end{equation*}
where $c_{n,k} \geq 0$, $\eps_{n,k} \in \{-1,1\}$, $(x_{n,k}, y_{n,k}) \in B_Z$ such that 
\begin{equation} \label{symmetric-eq6}
    \sum_{k=1}^{\infty} c_{n,k} \leq 1 + \frac{1 - \max_{(\eps,x,z) \in K \setminus U_n} \eps z(Ax)}{n}.
\end{equation}
Since $\Psi(v) = 1$, we have that 
\begin{equation*}
    1 = \sum_{k=1}^{\infty} c_{n,k} \eps_{n,k} y_{n,k}(A x_{n,k})
\end{equation*}
and then 
\begin{equation} \label{symmetric-eq7}
\sum_{k=1}^{\infty} c_{n,k}[1 - \eps_{n,k} y_{n,k}(A x_{n,k})] = \sum_{k=1}^{\infty} c_{n,k} - 1.
\end{equation}
For $(\eps_{n,k}, x_{n,k}, y_{n,k}) \not\in U_n$, we have that 
\begin{equation*}
1 - \eps_{n,k} y_{n,k}(A x_{n,k}) \geq 1 - \max_{(\eps,x,z) \in K \setminus U_n} \eps z(Ax) 
\end{equation*}
and equations (\ref{symmetric-eq6}) and (\ref{symmetric-eq7}) give
\begin{equation} \label{symmetric-eq8}
\sum_{\{k: (\eps_{n,k}, x_{n,k}, y_{n,k}) \not\in U_n \}} c_{n,k} \leq \frac{1}{n}.
\end{equation}
Thus, asymptotically, all the coefficients are concentrated near the finite number of points (notice that there are exactly 12 points) of $\mathcal{E}$. For $j=1,2,3$, $\eps, s \in \{-1,1\}$, let us define 
\begin{equation*}
    \lambda_{n, j, \eps, s} := \sum_{\{k: (\eps_{n,k}, x_{n,k}, y_{n,k}) \in U_{j,\eps,s}^n \}} c_{n,k}.
\end{equation*}
This sequence is bounded and so, passing to a subsequence if necessary, we can suppose that $\lambda_{n, j, \eps, s} \longrightarrow \lambda_{j, \eps, s}$ for every $j, \eps, s$. Since all the neighborhoods shrink to $(\eps, s x_j, \eps s z_j)$, all the coefficients outside $U_n$ tend to zero by (\ref{symmetric-eq8}). Applying $P_X$ to the representation of $v$ gives 
\begin{equation} \label{symmetric-eq9} 
P_X v = \sum_{j=1}^3 \left( \sum_{\eps, s = \pm 1} \eps \lambda_{j, \eps, s} \right) x_j^{\otimes 2}
\end{equation}
(notice that $s$ disappears since $(s x_j)^{\otimes 2} = x_j^{\otimes 2}$). In order to simplify the notation a bit, let us write 
\begin{equation*}
    \delta_j := \sum_{\eps, s = \pm 1} \eps \lambda_{j, \eps, s}
\end{equation*}
for every $j=1,2,3$. So, using (\ref{symmetric-eq9}), we can write 
\begin{equation} \label{symmetric-eq9-1} 
    P_X v = \delta_1 x_1^{\otimes 2} + \delta_2 x_2^{\otimes 2} + \delta_3 x_3^{\otimes 2}.
\end{equation}
We need to prove that $\delta_1 = \delta_2 = \delta_3 = 0$. In order to do so, take any quadratic polynomial $R$ on $Y$ depending only on finitely many coordinates and let $\tilde{R}$ denote its extension to $Y^{**}$. We have that 
\begin{equation*}
    P_Y v = \sum_{k=1}^{\infty} c_{n,k} \eps_{n,k} y_{n,k}^{\otimes 2}
\end{equation*}
and, therefore,
\begin{equation*}
    R(P_Y v) = \sum_{k=1}^{\infty} c_{n,k} \eps_{n,k} R(y_{n,k}).
\end{equation*}
Inside the neighborhoods $U_{j, \eps, s}^n$, we have that $y_{n,k} \stackrel{w^*}{\longrightarrow} \eps s z_j$ for $j=1,2,3$. The polynomial $R$ depends on only finitely many coordinates; therefore $R(y_{n,k}) \rightarrow \tilde{R}(\eps s z_j)$. Since $R$ is $2$-homogeneous, we get $\tilde{R}(\eps s z_j) = (\eps s)^2 \tilde{R}(z_j) = \tilde{R}(z_j)$ so, as before, we have 
\begin{equation*}
    R(P_Y v) = \sum_{j=1}^3 \sum_{\eps, s = \pm 1} \eps \lambda_{j, \eps, s} \tilde{R}(z_j).
\end{equation*}
This means we can write 
\begin{equation} \label{symmetric-eq10}
    R(P_Y v) = \delta_1 \tilde{R}(z_1) + \delta_2 \tilde{R}(z_2) + \delta_3 \tilde{R}(z_3).
\end{equation}
Now, choose, for $k \geq 3$, $R_k(y) = e_k^*(y)^2$ for $y \in c_0$. Since $y \in c_0$, $y(k) \rightarrow 0$ as $k \rightarrow \infty$ and then $R_k(y) \rightarrow 0$ as $k \rightarrow \infty$. As $(R_k)$ is uniformly bounded, it follows that $R_k(P_Y v) \rightarrow 0$ as $k \rightarrow \infty$. However, for $j=1,2,3$ and $k \geq 3$, $z_j(k) = 1$, we have that $\tilde{R}_k(z_j) = 1$. By (\ref{symmetric-eq10}), 
\begin{equation*}
    R_k(P_Y v) = \delta_1 + \delta_2 + \delta_3 
\end{equation*}
and since the left side goes to zero, we get that 
\begin{equation} \label{symmetric-eq11}
    \delta_1 + \delta_2 + \delta_3 = 0. 
\end{equation}
Next, take $R_{1,k}(y) := e_1^*(y) e_k^*(y)$ for $y \in c_0$. We know that $z_1(1) = 1$, $z_2(1) = z_3(1) = -1/2$ and $z_j(k) = 1$ for every $k \geq 3$. Therefore, $\tilde{R}_{1,k} (z_1) = 1$ while 
\begin{equation*}
    \tilde{R}_{1,k}(z_2) = \tilde{R}_{1,k}(z_3) = - \frac{1}{2}.
\end{equation*}
Again $R_{1,k}(P_Yv) \rightarrow 0$ as $k \rightarrow \infty$ and (\ref{symmetric-eq10}) gives 
\begin{equation} \label{symmetric-eq12}
\delta_1 - \frac{1}{2} \delta_2 - \frac{1}{2} \delta_3 = 0.
\end{equation}
Finally, taking $R_{2,k}(y) := e_2^*(y) e_k^*(y)$ for $y \in c_0$, we obtain
\begin{equation} \label{symmetric-eq13}
-\frac{1}{2} \delta_1 + \delta_2 - \frac{1}{2} \delta_3 = 0.
\end{equation} 
From (\ref{symmetric-eq11}), (\ref{symmetric-eq12}) and (\ref{symmetric-eq13}), we get that $\delta_1 = \delta_2 = \delta_3 = 0$. This shows that $P_X v = 0$ by (\ref{symmetric-eq9-1}). On the other hand, from (\ref{symmetric-eq10}) we have that $R(P_Y v) = 0$ for every finite-coordinate quadratic polynomial. These polynomials separate the points of $\polten{2}{Y}$, thus $P_Y v = 0$.
\end{proof}

Having all of this in mind, the proof of Theorem \ref{theorem:symmetric-main} is straightforward.

\begin{proof}[Proof of Theorem \ref{theorem:symmetric-main}] By Lemma \ref{symmetric1} $\Psi$ exposes $u_s$, therefore $u_s \in \Exp(B_{\polten{2}{Z}})$. Finally, suppose that $u_s = \pm w^{\otimes 2}$ for some $w = (x,y) \in Z$. Since $u_s \in M(\polten{2}{Z})$, we have that $u_s = Mu_s = \pm M((x,y)^{\otimes 2}) = \pm J(x \otimes y)$. Since $J$ is injective and $u_s = J(u)$, it follows that $u = \pm x \otimes y$, a contradiction.
\end{proof}


\section{Higher-order tensor products}\label{sec:higher-order}

The proof of Theorem \ref{theorem:higher-order-main} below is the main goal of this section, which extend both Theorem \ref{theorem:full-projective-main} and Theorem \ref{theorem:symmetric-main} from Sections \ref{sec:full tensor} and \ref{section:symmetric}, respectively. Let us notice that we have chosen to present the quadratic case first because the main ideas there are (much) more transparent than the proof of Theorem \ref{theorem:higher-order-main}. The passage to higher degrees is essentially a technical adaptation of the same argument and it might obscure the main mechanism.  

\begin{introthm}\label{theorem:higher-order-main} Let $N \geq 2$ be given. The following statements hold true.
\begin{itemize}
    \item[(a)] There exist real Banach spaces $X_1,\ldots,X_N$ such that the unit ball of $\multiten{N} X_i$ contains a non-elementary exposed point. 
    
    \item[(b)] There exists a real Banach space $F_N$ such that $B_{\polten{N}{F_N}}$ contains an exposed point which is not of the form $\pm z^{\otimes N}$.
    \end{itemize}
\end{introthm}

Once the example of Section~\ref{sec:full tensor} is available, it is relatively straightforward
to obtain analogous examples in projective tensor products with an
arbitrary number of factors. Indeed, by adding copies of $\ell_2^2$ and
considering $u\otimes e_1^{\otimes N-2}$, one obtains a non-elementary exposed tensor in $\multiten{N} X_i$ for every $N\geq3$, where $X_1=X$, $X_2=Y$ and $X_j=\ell_2^2$ for $j\geq3$. This proves item (a) in Theorem \ref{theorem:higher-order-main}.

The symmetric case is more delicate, since one needs to preserve both
the norm and the exposing structure when passing to higher degrees.
Nevertheless, the construction underlying Theorem \ref{theorem:symmetric-main}, based on ideas of Ansemil and Floret, can be adapted to obtain the following result as follows. 


We retain the notation from Sections~\ref{sec:full tensor} and
\ref{section:symmetric}. Fix $N\geq3$ and set
\begin{equation*}
F_N
:=
X\oplus_\infty Y\oplus_\infty\ell_\infty^{N-2}.
\end{equation*}
Let
$e_1,\ldots,e_{N-2}$ denote the canonical basis of the last component. Let us define now
$
J_N:X\otimes Y\longrightarrow \bigotimes^{N,s}F_N
$
by
\begin{equation*}
J_N(x\otimes y)
=
N!\,
(x,0,0)\vee(0,y,0)
\vee e_1\vee\cdots\vee e_{N-2}
\end{equation*}
for every $x \in X$ and $y \in Y$. We have the following result, which extends Fact \ref{symmetric:fact1}.

\begin{fact}\label{fact:JN}
The map $J_N$ extends to an isometric embedding
\begin{equation*}
J_N:X\pten Y\longrightarrow\polten{N}{F_N}
\end{equation*}
and its range is $1$-complemented.
\end{fact}

\begin{proof}
Let $B\in\mathcal B(X\times Y)$ and consider the $N$-homogeneous
polynomial
\begin{equation*}
P_B(x,y,t_1,\ldots,t_{N-2})
=
B(x,y)t_1\cdots t_{N-2}.
\end{equation*}
Clearly,
$
\|P_B\|=\|B\|,
$
and
$
\langle P_B,J_N(x\otimes y)\rangle
=
B(x,y).
$
By duality, we have that $\|J_N(w)\|_{\pi_s} \geq \|w\|_{\pi}$ for every $w \in X \otimes Y$. For the reverse inequality, let $x\neq0$ and $y\neq0$, and put
$
r :=(\|x\|\|y\|)^{1/N}.
$
Set also
$
\xi_1 :=\frac{r}{\|x\|}(x,0,0),
\xi_2 :=\frac{r}{\|y\|}(0,y,0),
$
and
$\xi_{k+2}=r e_k$ for every $k=1,\ldots,N-2$. Then, we have that
$
\|\xi_1\|=\cdots=\|\xi_N\|=r
$
and
\begin{equation*}
J_N(x\otimes y)
=
N!\,\xi_1\vee\cdots\vee\xi_N.
\end{equation*}
The polarization formula gives $\|J_N(x\otimes y)\|_{\pi_s}
\leq
r^N
=
\|x\|\|y\|$. Thus $J_N$ is an isometry.

To identify its range, we follow the ideas from Fact \ref{symmetric:fact1} by changing independently the signs of the $N$ blocks given by
$
X,
Y,
\R e_1,\ldots,\R e_{N-2}.
$
Averaging the induced isometries on $\polten{N}{F_N}$ with the product
of the signs gives a contractive projection
\begin{equation*}
M_N:\polten{N}{F_N}\longrightarrow\polten{N}{F_N}.
\end{equation*}
That is, 
\[
M_N
=
\frac1{2^N}
\sum_{\varepsilon\in\{-1,1\}^{m+2}}
\varepsilon_X\varepsilon_Y
\varepsilon_1\cdots\varepsilon_m\,
S_\varepsilon^{\otimes_{N,s}},
\]
where $S_\varepsilon(x,y,t_1,\ldots,t_m)
=
(\varepsilon_Xx,\varepsilon_Yy,
\varepsilon_1t_1,\ldots,\varepsilon_mt_m)$. The range of $M_N$ is precisely the component 
$
(1,1,\ldots,1),
$
which coincides with $J_N(X\pten Y)$.
\end{proof}

We now define $u_N :=J_N(u)$. Since $J_N$ is an isometry,
$
\|u_N\|_{\pi_s}=1.
$ Let $\varphi\in(X\pten Y)^*$ be the functional $u$ (recall the proof of Lemma \ref{lemma2}) and define
$
\Psi_N
:=
\varphi\circ J_N^{-1}\circ M_N.
$
As in Section \ref{section:symmetric}, we have
$
\|\Psi_N\|=1,
\Psi_N(u_N)=1
$
 and $\Psi_N(v)=1$ implies that  $M_N(v)=u_N$.

Now we prove the extension of Lemma \ref{symmetric1}. That is, we have the following result.

\begin{lemma}\label{lemma:higher-fiber}
For every $v\in B_{\polten{N}{F_N}}$ with $M_N(v)=u_N$, we have
$v=u_N$. In particular,
\begin{equation*}
\left\{
v\in B_{\polten{N}{F_N}}:\Psi_N(v)=1
\right\}
=
\{u_N\}.
\end{equation*}
\end{lemma}

\begin{proof}
Put $m=N-2$. For
$
w :=(x,y,t_1,\ldots,t_m)\in F_N,
$
let us write $w_X := (x,0,0)$, $w_Y:= (0,y,0)$ and $w_k:= t_k e_k$ for every $k=1,\ldots, m$. If $\alpha :=(a,b,c_1,\ldots,c_m)$ with $|\alpha|:=a+b+c_1+\cdots+c_m=N$, we write
\begin{equation*}
w^\alpha
=
w_X^{\otimes a}
\vee
w_Y^{\otimes b}
\vee
w_1^{\otimes c_1}
\vee\cdots\vee
w_m^{\otimes c_m}.
\end{equation*}
Then, we have that
\begin{equation*}
w^{\otimes N}
=
\sum_{|\alpha|=N}
\binom{N}{\alpha}w^\alpha \ \ \mbox{with} \ \ 
\binom{N}{\alpha}
=
\frac{N!}{a!b!c_1!\cdots c_m!}.
\end{equation*}
Let $P_\alpha$ denote the corresponding projection, that is,
\begin{equation*}
P_\alpha(w^{\otimes N})
=
\binom{N}{\alpha}w^\alpha.
\end{equation*}
Thus
\begin{equation*}
\id_{\polten{N}{F_N}}
=
\sum_{|\alpha|=N}P_\alpha,
\qquad
M_N=P_{(1,\ldots,1)}.
\end{equation*}
Let
$
v\in B_{\polten{N}{F_N}}
$ with $
M_N(v)=u_N.
$ We want to see that $v=u_N$.
Since
\begin{equation*}
v
=
u_N+
\sum_{\substack{|\alpha|=N\\
\alpha\neq(1,\ldots,1)}}
P_\alpha(v),
\end{equation*}
it remains to prove that
$
P_\alpha(v)=0
$ for $
\alpha\neq(1,\ldots,1)$. For every $n\in\N$, choose an almost optimal representation
\begin{equation*}
v
=
\sum_{k=1}^{\infty}
c_{n,k}\eps_{n,k}
w_{n,k}^{\otimes N},
\end{equation*}
where
$
c_{n,k}\geq0,
$
 $\eps_{n,k}\in\{-1,1\},
$
$
w_{n,k}\in B_{F_N}
$
and
\begin{equation*}
\sum_{k=1}^{\infty}c_{n,k}
\leq
1+\frac1n.
\end{equation*}
Write
$
w_{n,k}
=
(x_{n,k},y_{n,k},
t_{n,k,1},\ldots,t_{n,k,m}).$ Consider once again the set
\begin{equation*}
K
:=
\{-1,1\}
\times B_X
\times B_{Y^{**}}
\times[-1,1]^m
\end{equation*}
where $B_X$ has the norm topology and $B_{Y^{**}}$ the weak-star
topology. Here, instead of using $w^*$-neighborhoods as in the proof of Lemma \ref{symmetric1}, we use measures. We define the positive measure
\begin{equation*}
\mu_n
=
\sum_{k=1}^{\infty}
c_{n,k}
\delta_{
(\eps_{n,k},x_{n,k},y_{n,k},
t_{n,k,1},\ldots,t_{n,k,m})
}.
\end{equation*}
After passing to a subsequence,
$
\mu_n\xrightarrow{w^*}\mu
$
for some probability measure $\mu$ on $K$. Since $\Psi_N(v)=1$, if we define $h(\eps,x,z,\tau_1,\ldots,\tau_m)
:=
\eps\,z(Ax)\tau_1\cdots\tau_m$, then
\[
\int_K h\,d\mu_n
=
\sum_{k=1}^\infty
c_{n,k}\varepsilon_{n,k}\,
y_{n,k}(Ax_{n,k})
\prod_{\ell=1}^m t_{n,k,\ell}
=
\Psi_N(v)=1.
\]
Passing to the limit gives
\[
\int_K h\,d\mu=1.
\]
Since $h\leq1$ on $K$ and $\mu$ is a probability measure, it follows that
$h=1$ on $\supp\mu$. If $h(\eps,x,z,\tau)=1$, then necessarily
$|\tau_\ell|=1$ for every $\ell$ and $|z(Ax)|=1$. By Fact
\ref{fact3} and Lemma \ref{lemma1},
\(
x=sx_j
\) and $
z=\eps s\tau_1\cdots\tau_m z_j
$
for some $j\in\{1,2,3\}$ and $s\in\{-1,1\}$. Then, we have that
\begin{equation*}
\operatorname{supp}\mu
\subseteq
\left\{
\left(
\eps,
sx_j,
\eps s\tau_1\cdots\tau_m z_j,
\tau_1,\ldots,\tau_m
\right):
j=1,2,3,\
\eps,s,\tau_k\in\{-1,1\}
\right\}.
\end{equation*}
For $
j\in\{1,2,3\}
$ and $
\eps,s,\tau_1,\ldots,\tau_m\in\{-1,1\},
$
let us put
\begin{equation*}
m_{j,\eps,s,\tau}
=
\mu\left(
\left\{
\left(
\eps,
sx_j,
\eps s\tau_1\cdots\tau_m z_j,
\tau_1,\ldots,\tau_m
\right)
\right\}
\right).
\end{equation*}

Now, for every
$
\beta=(r,q,d_1,\ldots,d_m)
$ with $
|\beta|=N$
define
\begin{equation}\label{eq:delta-beta}
\delta_j^\beta
=
\sum_{\eps,s,\tau}
\eps^{q+1}s^{r+q}
\prod_{\ell=1}^m
\tau_\ell^{q+d_\ell}
m_{j,\eps,s,\tau},
\qquad
j=1,2,3.
\end{equation}
If $Q$ is an $N$-homogeneous polynomial having degrees
$r,q,d_1,\ldots,d_m$ in the corresponding variables and
depending only on finitely many coordinates of $Y$, let
$\widetilde Q$ denote its canonical extension to
\(
X\oplus_\infty Y^{**}\oplus_\infty\ell_\infty^m.
\)
Then
\begin{equation}\label{eq:measure-component}
\langle Q,v\rangle
=
\sum_{j=1}^3
\delta_j^\beta
\widetilde Q(x_j,z_j,1,\ldots,1).
\end{equation}
Indeed,
\begin{eqnarray*}
\langle Q,v\rangle
&=& \sum_{k=1}^{\infty}
c_{n,k}\eps_{n,k}
Q(w_{n,k})=  \sum_{k=1}^{\infty}
c_{n,k}\eps_{n,k}
\tilde Q(w_{n,k})\\
&=& \int_K \eps \tilde Q (x,y,\tau_1,\ldots, \tau_m)d \mu_n(\eps, x,y,\tau_1,\ldots, \tau_m).
\end{eqnarray*}
Taking the limit as $n\to\infty$, we obtain
\begin{eqnarray*}
\langle Q,v\rangle &=& \int_K \eps \tilde Q (x,y,\tau_1,\ldots, \tau_m)d \mu(\eps, x,y,\tau_1,\ldots, \tau_m) \\
&=& \sum_{j,\eps,s,\tau}
m_{j,\eps,s,\tau}
\eps\tilde Q(sx_j,
\eps s\tau_1\cdots\tau_m z_j,
\tau_1,\ldots,\tau_m)\\
&=& \sum_{j,\eps,s,\tau}
\eps^{q+1}s^{r+q}
\prod_{\ell=1}^m
\tau_\ell^{q+d_\ell}
m_{j,\eps,s,\tau}
 \tilde Q(x_j,
 z_j,
1,\ldots,1)\\
&=& \sum_{j=1}^3\left(\sum_{\eps,s,\tau}
\eps^{q+1}s^{r+q}
\prod_{\ell=1}^m
\tau_\ell^{q+d_\ell}
m_{j,\eps,s,\tau}
 \right)\tilde Q(x_j,
 z_j,
1,\ldots,1)\
\end{eqnarray*}
Thus, to see that $P_\beta(v)=0$, it is enough to prove that $\delta_j^\beta=0$ for $j=1,2,3.$

Fix now
$
\alpha=(a,b,c_1,\ldots,c_m)$ with $
|\alpha|=N$ and $
\alpha\neq(1,\ldots,1).
$
If $b\geq2$, put $\alpha'=\alpha$. Otherwise, at least one of
$a,c_1,\ldots,c_m$ is greater than or equal to $2$. Subtract $2$ from
one such exponent and add $2$ to $b$, obtaining
$
\alpha'=(a',b',c'_1,\ldots,c'_m)
$
with
$
b'\geq2.
$
Since only even changes have been made,
$
\delta_j^\alpha
=
\delta_j^{\alpha'}.
$ The key idea in this part is that degree can be transferred from one of the other indices to the $Y$-variable, increasing $b$ until it is at least $2$.

Put
$
q=b'\geq2.
$
If $a'>0$, choose $x^*\in X^*$ such that
$
x^*(x_j)\neq0
$ for $
j=1,2,3.
$
For every finite-coordinate polynomial $R\in\mathcal P(^qY)$, set
\begin{equation*}
Q_R(x,y,t_1,\ldots,t_m)
=
x^*(x)^{a'}R(y)
\prod_{\ell=1}^m t_\ell^{c'_\ell},
\end{equation*}
where the factor involving $x^*$ is omitted when $a'=0$.
By \eqref{eq:measure-component},
\begin{equation*}
\langle Q_R,v\rangle
=
\sum_{j=1}^3
\gamma_j\widetilde R(z_j),
\end{equation*}
where
$
\gamma_j
=
\delta_j^{\alpha'}x^*(x_j)^{a'},
$
with the convention $x^*(x_j)^0=1$. For $k\geq3$, let us take $R_k=(e_k^*)^q$, $R_{1,k}=e_1^*(e_k^*)^{q-1}$ and $R_{2,k}=e_2^*(e_k^*)^{q-1}$. These sequences are uniformly bounded and converge pointwise to zero
on $Y=c_0$. Hence, $\langle Q_{R_k},v\rangle\rightarrow0$, $\langle Q_{R_{1,k}},v\rangle\rightarrow0$ and $\langle Q_{R_{2,k}},v\rangle\rightarrow0$. Exactly as in Lemma~\ref{symmetric1}, we obtain $\gamma_1+\gamma_2+\gamma_3=0$, $\gamma_1-\frac12\gamma_2-\frac12\gamma_3=0$ and $-\frac12\gamma_1+\gamma_2-\frac12\gamma_3=0$. 
Thus, $\gamma_1=\gamma_2=\gamma_3=0$ and consequently $\delta_j^\alpha
=
\delta_j^{\alpha'}
=
0$ for $j=1,2,3$. By \eqref{eq:measure-component}, this implies that $
P_\alpha(v)=0,
$ as we wanted to see.
\end{proof}

\begin{proof}[Proof of Theorem \ref{theorem:higher-order-main}.(b)] By Lemma \ref{lemma:higher-fiber}, we have that $\Psi_N$ exposes $u_N$. This means that $u_N \in \Exp(B_{\polten{N}{F_N}})$. Finally, as in the proof of Theorem \ref{theorem:symmetric-main}, let us suppose that $u_N = \pm w^{\otimes N}$ for some $w = (x,y,t_1,\ldots,t_m) \in F_N$. Since $u_N \in M_N(\polten{N}{F_N})$, we have that $u_N = M_N(u_N) = \pm M_N(w^{\otimes N}) = \pm t_1\cdots t_m J_N(x \otimes y)$. Since $J_N$ is injective and $u_N = J_N(u)$, it follows that $u = \pm t_1\cdots t_mx \otimes y$. This contradiction yields the result and we are done.
\end{proof}

\noindent
\textbf{Funding.} S. Dantas has been supported by the grants PID2021-122126NB-C31 and PID2021-122126NB-C33 funded by MICIU/AEI/ 10.13039/ 501100011033 and by ERDF/EU. M. Mazzitelli has been partially supported by CONICET PIP  1122020 0102366CO, SIIP 80020240100196UN, UNCOMA PIN I 04/B251.  J. T. Rodríguez has been partially supported by CONICET PIP 11220200101609CO. 

\noindent
\textbf{Use of generative AI.}
The original ideas and constructions developed in this paper are due to the authors. OpenAI's ChatGPT was used during the development of the work to assist with some technical aspects of the arguments in both the projective and symmetric cases. In particular, it was especially useful in obtaining the first versions of the proofs of Lemmas \ref{lemma1} and \ref{symmetric1}. These proofs were subsequently carefully studied, revised, digested and substantially simplified by the authors. It also suggested the norm on $X$ used in the final version of the construction; in an earlier draft, the authors worked with a Hilbertian norm and a different choice of the vectors $x_1,x_2,x_3$, which led to less transparent computations and obscured some of the underlying ideas. The authors take full responsibility for the mathematical content of the paper.

\end{document}